\documentclass{birkjour}
\usepackage{graphicx}
\usepackage[ansinew]{inputenc}    
\usepackage{fancyhdr}

\usepackage{color}

\usepackage{amsmath,amstext,amssymb,}
\usepackage{mathrsfs,amscd}

\usepackage[colorlinks]{hyperref}
\definecolor{wine-stain}{rgb}{0.5,0,0}
\hypersetup{
  colorlinks,
  linkcolor=wine-stain,
  filecolor=red,
  urlcolor=green,
  citecolor=blue
}

\newtheorem{thm}{Theorem}[section]
\newtheorem{cor}[thm]{Corollary}
\newtheorem{lem}[thm]{Lemma}
\newtheorem{prop}[thm]{Proposition}

\renewcommand{\proofname}{Proof}

\theoremstyle{definition}

\theoremstyle{definition}
\newtheorem{rem}[thm]{Remark}

\theoremstyle{definition}
\newtheorem{exam}[thm]{Example}

\def\R{\mathbb R}

\def\e{\mathbf e}

\begin{document}

\title[Equiaffine Darboux Frames for Codimension $2$ Submanifolds]{Equiaffine Darboux Frames for Codimension $2$ Submanifolds contained in Hypersurfaces}
{\color{red}
\author[M.Craizer]{Marcos Craizer}
\address{%
Pontif\'icia Universidade Cat\'olica do Rio de Janeiro, Departamento de Matem\'atica, 22453-900 Rio de janeiro (RJ), BRAZIL}
\email{craizer@puc-rio.br}

\author[M.J.Saia]{Marcelo J. Saia}
\address{Universidade de S\~ao Paulo, ICMC-SMA, Caixa Postal 668,
13560-970 S\~ao Carlos (SP), BRAZIL}
\email{mjsaia@icmc.usp.br}

\author[L.F.Sánchez]{Luis F. Sánchez}
\address{Universidade Federal de Uberlândia, FAMAT, Departamento de Matemática,  Rua Goiás 2000, 38500-000 Monte Carmelo (MG),  BRAZIL}
\email{luis.sanchez@ufu.br}

\thanks{The first author wants to thank CNPq for financial support during the preparation of this manuscript. The second and third authors  have been partially supported by CAPES, FAPESP and CNPq.}
}



\date{December 23, 2014}

\begin{abstract}
Consider a codimension $1$ submanifold $N^n\subset M^{n+1}$, where $M^{n+1}\subset\R^{n+2}$ is a hypersurface. The envelope of tangent spaces of $M$ along $N$ generalizes 
the concept of tangent developable surface of a surface along a curve. In this paper, we study the singularities of these envelopes. 

There are some important examples of submanifolds that admit a vector field tangent to $M$ and transversal to $N$ whose derivative in any direction of $N$ is contained in $N$.
When this is the case, one can construct transversal plane bundles and affine metrics on $N$ with the desirable properties of being equiaffine and apolar. Moreover,
this transversal bundle coincides with the classical notion of Transon plane. But we also give an explicit example of a submanifold that do not admit a vector field with the above property.
\end{abstract}

\maketitle

\section{Introduction}

Consider a surface $M\subset\R^3$ and let $\gamma:I\to\R$ be a smooth curve, where $I\subset\R$ is an interval. Under some general hypothesis, one can find a
unique, up to sign, vector field $\xi$ tangent to $M$ along $\gamma$ such that $\xi'(t)$ is tangent to $\gamma$, for any $t\in I$. The developable surface
$$
OD_{\gamma}(t,u)=\gamma(t)+u\xi(t),\ \ t\in I,\ \  u\in\R,
$$
is called the {\it tangent developable} of $M$ along $\gamma$ and has been extensively studied. The natural Darboux frame to consider here
is $\{\gamma'(t),\gamma''(t),\xi(t)\}$, where the parameterization $\gamma(t)$ satisfies $\gamma'''(t)$ tangent to $M$. 
Writing $\xi'(t)=-\sigma(t)\gamma'(t)$, we have that, 
from the point of view of singularity theory, the behavior of the $OD_{\gamma}$ in a neighborhood of $u=\sigma^{-1}(t)$ depends basically of the number of vanishing derivatives of
$\sigma$ (\cite{Izu2},\cite{Izu3}).

In this paper, we want to generalize this construction to arbitrary codimension $1$ submanifolds $N^n\subset M^{n+1}$, where $M^{n+1}\subset\R^{n+2}$ is a hypersurface.
As we shall see, under general hypothesis there exists a vector field $\xi$ tangent to $M$ along $N$ such that $D_X\xi$ is tangent to $M$, for any $X$ tangent to $N$. The direction determined by $\xi$ is unique and we call it the {\it (osculating) Darboux direction}. The hypersurface
$$
ET_N(p,u)=p+u\xi(p), \ \ p\in N,\ \ u\in\R,
$$
generalizes the tangent developable surface of a surface along a curve and we shall call {\it envelope of tangent spaces} of $M$ along $N$. We shall verify that $ET_N$ is smooth
when $u$ is not the inverse of a non-zero eigenvalue of the shape operator associated with $\xi$. The singularities of $ET_N$ are also studied: We show by examples that all simple
singularities are realizable.

For a fixed vector field $\xi$ in the osculating Darboux direction, we can define a metric $g=g_{\xi}$ on $N$ and a distinguished transversal plane bundle $\sigma=\sigma(\xi)$, the {\it affine normal plane bundle}. It is natural to consider a basis $\{\xi,\eta\}$ of $\sigma$ such that $[X_1,...,X_n,\eta,\xi]=1$, where $\{X_1,...,X_n\}$ is a $g_{\xi}$-orthonormal basis of $TN$. But this condition does not determine uniquely the vector $\eta$. In fact, any $\bar\eta=\eta+\lambda\xi$ does the same job. 

The main difficulty here is to choose 
a "good" vector field $\xi$ in the Darboux direction. 
In the case of curves, we can choose a vector field $\xi$ in the Darboux direction such that $D_X\xi$ is tangent to $N$, for any $X$ tangent to $N$. We shall refer to this latter property by saying that $\xi$ is {\it parallel}. On the other hand, when $M$ is non-degenerate, we can choose
$\xi$ such that $g_{\xi}$ coincides with the restriction of the Blaschke metric of $M$ to $N$ and the affine Blaschke normal belongs to the affine normal plane bundle. But this choice of $\xi$ is parallel only in very special cases.

The metric $g_{\xi}$ and the affine transversal plane bundle $\sigma(\xi)$ have more desirable properties when $\xi$ is parallel. 
Denote by  $\omega=\omega(g_{\xi})$ the volume form on $N$ determined by the metric $g_{\xi}$ and by $\nabla=\nabla(\xi)$ the
connection determined by $\sigma$.  The pair $(\nabla,g)(\xi)$ is equiaffine if $\nabla(\omega_g)=0$. We can also define the cubic form $C^2$ and the metric $h^2$ relative to
the vector field $\eta=\eta(\xi)$. We say that $(C^2,h^2)$ is apolar if $tr_{h^2}C^2(X,\cdot,\cdot)=0$, for any $X$ tangent to $N$. We shall verify that the properties
$(\nabla,g)$ equiaffine and $(C^2,h^2)$ apolar are both equivalent to the parallelism of $\xi$.

Consider now hyperplanes $H$ containing the tangent space of $N$. The intersection of $H$ with $M$ determine a codimension $1$ submanifold of the hyperplane $H$ and thus we can consider its Blaschke normal vector $\eta(H)$. When we vary $H$, the vector $\eta(H)$ describes a $2$-plane called the {\it Transon plane} of $T_pN$ with respect to $M$. In the case of curves, this is a very classic result of A.Transon (\cite{Transon}), see \cite{Juttler} for a modern reference. We shall verify that the Transon plane coincides with the affine normal plane if and only if $\xi$ is parallel.

The latter two paragraphs show that the condition of $\xi$ parallel is very significant. So it is natural to ask whether, for a given immersion $N\subset M$, a parallel vector field $\xi$ exists or not. There are several examples of immersions $N\subset M$ that admit a parallel vector field: Curves in surfaces, submanifolds contained in hyperplanes, visual contour submanifolds, submanifolds contained in hyperquadrics. But there are also examples of immersions that do not admit parallel vector fields, and we give explicitly such an example. We prove also that the existence of a parallel vector field is equivalent to the flatness of the affine normal bundle of the immersion $N\subset M$.

 The paper is organized as follows: In section 2  we discuss the generalization of the osculating Darboux direction
 and tangent developable surfaces to codimension $2$ submanifolds contained in hypersurfaces. In section 3, we study the singularities of these hypersurfaces. In section 4, we recall the constructions of the affine metric and the affine normal plane bundle associated with a vector field $\xi$ in the osculating Darboux direction. In section 5 we define the parallelism condition of $\xi$ and show the equivalence of this property and the equiaffinity of $(\nabla,g)$ and the apolarity of $(C^2,h^2)$. We also give some important examples
of immersions $N\subset M$ with parallel vector fields. In section 6 we recall the notion of Transon planes and prove that it coincides with the affine normal plane if and only if $\xi$ is parallel. Finally in section 7 we show an example of an immersion $N\subset M$ that does not admit a parallel vector field and prove that the existence of a parallel vector field is equivalent to the flatness of the affine normal bundle connection.

\section{Darboux directions and the envelope of tangent spaces}

Consider a codimension $1$ immersion $N^n\subset M^{n+1}$, where $M^{n+1}\subset\R^{n+2}$ is a hypersurface.

\subsection{Basic equations}

Fix vector fields $\eta$ transversal to $M$ and $\xi$ tangent to $M$ transversal to $N$. For $X,Y$ vector fields tangent to $N$, we write
\begin{equation}\label{eq:Decomposition}
D_XY=\nabla_XY+h^1(X,Y)\xi+h^2(X,Y)\eta,
\end{equation}
where $\nabla_XY$ is tangent to $N$. It is straightforward to verify that $\nabla$ is a torsion-free connection on $N$ and $h^i$, $i=1,2$ are bilinear symmetric forms. 
Write also
\begin{equation}
\nabla_{X_i}X_j=\sum_{k=1}^n \Gamma_{ij}^k X_k, \  \  1\leq i,j\leq n,
\end{equation}
where $\Gamma_{ij}^k$ are the Christoffel symbols of the connection.

The derivatives of $\eta$ and $\xi$ can be written as
\begin{equation}\label{eq:DXXi}
\begin{array}{c}
D_X\xi=-S_1X+\tau_1^1(X)\xi+\tau_1^2(X)\eta,\\
D_X\eta=-S_2X+\tau_2^1(X)\xi+\tau_2^2(X)\eta,
\end{array}
\end{equation}
where $S_i$, $i=1,2$, are $(1,1)$-tensors of $N$ called {\it shape operators} and $\tau_i^j$ are $1$-forms on $N$.

\subsection{Osculating Darboux direction}\label{sec:Xi}

In this section, we generalize the notion of osculating Darboux direction from curves $\gamma\subset M$ to codimension $1$ submanifolds $N\subset M$.
Given a local frame $\{X_1,...,X_n\}$ of $TN$, we say that the immersion $N\subset M\subset\R^{n+2}$ is {\it non-degenerate} if the matrix
$h^2(X_i,X_j)$ is non-degenerate.

\begin{lem}
The non-degeneracy condition is independent of the choice of the local frame $\{X_1,...,X_n\}$ of $TN$, of the vector field $\xi$ tangent to $M$ and of the transversal vector field $\eta$.
\end{lem}
\begin{proof}
Suppose we fix $\xi$ and $\eta$ and let $\{Y_1,...,Y_n\}$ be a local frame of $TN$. Then we can write
$$
Y_i=\sum_{i=1}^n a_{ij}X_j
$$
for a certain invertible matrix $A=(a_{ij})$. It is not difficult to verify that $\left( h^2(Y_i,Y_j) \right)=A\left( h^2(X_i,X_j) \right)A^{t}$, thus proving that the non-degeneracy condition
is invariant by a change of basis of $TN$. If we change $\xi$ by $\bar\xi$ satisfying
$$
\xi=\sum b_{k}X_k+\beta\bar\xi,
$$
then $\bar{h^2}(X_i,X_j)=h^2(X_i,X_j)$ and so the non-degeneracy condition is invariant by the choice of the vector $\xi$. Finally if we write
$$
\eta=\sum b_kX_k+\beta\xi+\gamma\bar\eta,
$$
then $\bar{h^2}(X_i,X_j)=\gamma h^2(X_i,X_j)$, thus completing the proof of the lemma.
\end{proof}

\begin{prop}\label{prop:ExistenceXi}
Assume that the immersion $N^n\subset M^{n+1}\subset\R^{n+2}$ is non-degenerate.
There exist a unique direction $\xi$ tangent to $M$ along $N$ and transversal to $N$ such that $D_X\xi$ is tangent to $M$, for any $X\in T_pN$. We shall call this direction
the {\it osculating Darboux direction} of $N\subset M$.
\end{prop}

\begin{proof}
We first remark that if $D_X\xi$ is tangent to $M$ for any $X\in T_pM$, the same holds for $\lambda\xi$, for any $\lambda:N\to\R$. Take any $\xi_1$ tangent to $M$ and write
$$
\xi=\sum_{j=1}^n\alpha_j X_j+\xi_1.
$$
Then the component of $D_{X_i}\xi$ in the direction $\eta$ is
$ \sum_{j=1}^n\alpha_jh^2(X_i,X_j)+ \tau_1^2(X_i) $.
Thus we have to solve the system
$$
\sum_{j=1}^n \alpha_jh^2(X_i,X_j)+ \tau_1^2(X_i)=0,\ \  1\leq i\leq n,
$$
which admits a unique solution by the non-degeneracy hypothesis.
\end{proof}

\begin{rem}
In the case of curves, the non-degeneracy hypothesis is equivalent to $\gamma''(t)\not\in T_{\gamma(t)}M$, i.e., the osculating plane of $\gamma$ does not coincide 
with the tangent plane of $M$.
\end{rem}

Next example shows that the non-degeneracy hypothesis is necessary:

\begin{exam}
Consider $M$ given by $\psi(u,v)=(u,v,uv)$ and $N$ given by $\gamma(v)=(0,v,0)$. Any tangent vector field along $\gamma$ can be written as $B(v)=a(1,0,v)+b(0,1,0)$.
Thus $B'(v)=(a'(v), b'(v), a(v)+a'(v)v)$. We conclude that $B'(0)=(a'(0),b'(0), a(0))$ is tangent to $M$ if and only if $a(0)=0$. But then $B$ is tangent to $\gamma$.
\end{exam}

\begin{rem}\label{rem:BlaschkeHyp}
The hypersurface $M\subset\R^{n+2}$ is called {\it non-degenerate} if the $(n+1)\times(n+1)$ matrix
$(h(X_i,X_j))$ is invertible, where
$$
D_{X_i}X_j={\tilde\nabla}_{X_i}X_j+h(X_i,X_j)\eta,
$$
${\tilde\nabla}_{X_i}X_j$ is tangent to $M$ and $X_{n+1}=\xi$. In this case, the osculating Darboux direction $\xi$ is $h$-orthogonal to the tangent space of $N$. In fact,
since $h(X,\xi)=0$, for any $X\in TN$, we have that $D_{X}\xi=\tilde\nabla_XX_{n+1}$ is tangent to $M$.
\end{rem}

\begin{rem}
To define the osculating Darboux direction $\xi$ we need only to know the tangent space to $M$ at each point of $N$. Thus, instead starting with a codimension $2$ submanifold
$N$ contained in a hypersurface $M$, we could also have started with a codimension $2$ submanifold $N$ together with a hyperspace bundle containing the tangent space of $N$, 
without an explicit mention to $M$.
\end{rem}


\subsection{Envelope of Tangent Spaces of $M$ along $N$}\label{sec:TangentHyp}

Consider a curve $\gamma\subset$ where $M$ is a surface in $\R^3$. The surface 
\begin{equation}
OD_{\gamma}(t,u)=\gamma(t)+u\xi(t)
\end{equation}
is called the {\it developable tangent surface} of $M$ along $\gamma$ and has been studied by many authors (\cite{Izu2},\cite{Izu3}).

We can generalize the above definition to arbitrary dimensions. Let $\phi:U\subset\R^n\to \R^{n+2}$ be a parameterization of $N$ and define the {\it envelope of tangent spaces} of $M$ along $N$ by
\begin{equation}\label{eq:TangentHyp}
ET_{N}(t,u)=\phi(t)+u\xi(t)
\end{equation}
for $t\in U$ and $\xi(t)$ in the osculating Darboux direction.

Denote by $\{X_1,...,X_n\}$ the local frame of $N$ given by $X_i=D_{t_i}\phi$.
The hypersurface $ET_N$ can be studied by considering $F:U\times\R^{n+2}\to\R$ defined by
\begin{equation}\label{eq:DefineFH}
F(t,x)=\left[  X_1(t),...,X_n(t), \xi(t), x-\phi(t)  \right].
\end{equation}
Observe that $F=0$ is the equation of the tangent space of $M$ at a point of $N$. The {\it discriminant set} or {\it envelope} of $F$ is defined by
\begin{multline}
\mathcal{D}_F=\{x\in\R^{n+2}| F(t,x)=F_{t_1}(t,x)=...=F_{t_n}(t,x)=0, \text{for some}\\ t=(t_1,...,t_n)\in U\}.
\end{multline}
Next lemma justifies the name envelope of tangent spaces given to $ET_N$. 

\begin{lem}
The envelope $\mathcal{D}_F$ of $F$ coincides with the hypersurface defined by equation \eqref{eq:TangentHyp}. 
\end{lem}
\begin{proof}
Observe that
\begin{equation}
F_{t_i}(t,x)=a_i(t) F(t,x) +\sum_{l=1}^n h^2(X_i,X_l)\left[X_1,...,\eta,...,X_n, \xi, x-\phi(t) \right],
\end{equation}
where $a_i=\tau_1^1(X_i)+\sum_{l=1}^n\Gamma_{li}^i$, and $\eta$ is placed in the coordinate $l$ in the second parcel of the second member. 
Since the matrix $(h^2(X_i,X_j))$ is non-degenerate, $F=F_{t_1}=...=F_{t_n}=0$ if and only if $F=G_1=...=G_n=0$, where
\begin{equation}\label{eq:DerivativeFH}
G_l=\left[X_1,...,\eta,...,X_n, \xi, x-\phi(t) \right].
\end{equation}
This implies $x-\phi(t)=u\xi$, for some $u\in\R$, thus proving the lemma.
\end{proof}

The regression points of $F$ are points of its discriminant set where its hessian $D_{tt}F$ is degenerate.

\begin{lem}
The regression points of $F$ correspond to points where $u$ is the inverse of some non-zero eigenvalue of $S_1$.
\end{lem}
\begin{proof}
We may assume that $[X_1,...,X_n,\eta,\xi]=1$. 
At a point of the discriminant set of $F$, we have $F=G_1=...=G_n=0$. Using that the matrix $(h^2(X_i,X_j))$ is non-degenerate, 
it is not difficult to see that, at $\mathcal{D}(F)$, the matrix $(F_{t_it_j})$ is degenerate if and only if the matrix $((G_i)_{t_j})$ is degenerate.
Differentiating equation \eqref{eq:DerivativeFH} we obtain that, at $\mathcal{D}(F)$, 
$$
(G_i)_{t_j}=-\left[X_1,...,\eta,...,X_n, S_1(X_j), x-\phi(t) \right]+\left[X_1,...,\eta,...,X_n, \xi, -X_j \right]
$$
where $\eta$ is placed in coordinate $i$. Thus, at these points, 
$$
(G_i)_{t_j}= us_{ij}-\delta_{ij}
$$
where $s_{ij}$ is the $(i,j)$-entry of the matrix of $S_1$ in basis $\{X_1,...,X_n\}$. We conclude that the matrix $((G_i)_{t_j})$ at $\mathcal{D}(F)$
is exactly $uS_1-Id$, thus proving the lemma.
\end{proof}

Next corollary gives condition for the smoothness of $ET_N$:

\begin{cor}
If $u$ is not the inverse of a non-zero eigenvalue of $S_1$, the hypersurface $ET_N$ is smooth.
\end{cor}
\begin{proof}
Consider the map $G:U\times\R^{n+2}\to\R^{n+1}$ given by
$$
G(t,x)=\left( F(t,x), F_{t_1}(t,x),....,F_{t_n}(t,x) \right).
$$
Then $ET_N=G^{-1}(0)$ and we shall verify that $0$ is a regular value of $G$. But
\[
DG=\left[
\begin{array}{cc}
0 & F_x\\
F_{tt} & F_{xt}
\end{array}
\right].
\]
By the above lemma, $F_{tt}$ is non-degenerate. On the other hand, by considering derivatives in the direction $\eta$ one easily verifies that $F_x\neq 0$.
This shows that $0$ is a regular value of $G$,  thus proving the corollary.
\end{proof}

\section{Singularities of the envelope of tangent spaces}

In this section we study the singularities of $ET_N$. We begin with the case of curves, where a complete classification is given. 
For the general case, we show by examples that any simple singularity can occur. 

\subsection{Singularities of the tangent developable surface}

Let $M\subset\R^3$ be a surface and $\gamma:I\to M$ a smooth curve. 
Denote $S_1\gamma'(t)=-\sigma(t)\gamma'(t)$ and $S_2\gamma'(t)=-\mu(t)\gamma'(t)$, where $\eta(t)=\gamma''(t)$. We may assume 
that $[\gamma'(t),\eta(t),\xi(t)]=1$, for any $t\in I$, which implies that $\tau_1^1+\tau_2^2=0$. The Frenet equations are then
\[
\left\{
\begin{array}{c}
(\gamma')'= \eta\\
\eta'=-\mu\gamma'-\tau_1^1\eta+\tau_2^1\xi\\
\xi'=-\sigma\gamma'+\tau_1^1\xi.
\end{array}
\right. 
\]

\smallskip\noindent
Next proposition is proved in \cite{Izu2} using euclidean invariants. We give here a proof using affine invariants. 

\begin{prop}\label{prop:Sing}
Let $\gamma:I\to M$ be a smooth curve and $t_0\in I$ with $\sigma(t_0)\neq 0$.  For $u_0=\sigma^{-1}(t_0)$, we have that, at $OD_{\gamma}(t_0,u_0)$, 
\begin{enumerate}
\item $OD(\gamma)$ is locally diffeomorphic to a cuspidal edge if $[\sigma_t-\sigma\tau_1^1](t_0)\neq 0$.
\item $OD(\gamma)$ is locally diffeomorphic to a swallowtail if $[\sigma_t-\sigma\tau_1^1](t_0)=0$ and $[\sigma_t-\sigma\tau_1^1]_t(t_0)\neq 0$.
\end{enumerate}
\end{prop}
 
 \begin{rem}
 We shall see in section \ref{sec:Curves} that it is possible to parameterize $\gamma$ such that $\tau_1^1=0$. With such a parameterization, the formulas
 of the above proposition become much simpler. 
 \end{rem}

\smallskip\noindent
We shall need a well-known result from singularity theory (\cite{Giblin},\cite{Izu2}). 

\begin{lem}\label{lem:unfolding} Let $F:I\times\R^r \to \R$ denote a $r$ parameter unfolding
of $f(t)=F(t,x_0)$. Assume that $f(t)$ has an $A_k$-singularity at $t=t_0$. The unfolding $F(t,x)$ is $\mathcal{R}$-versal if the $k\times r$ matrix $j^{k-1}F_x$ has rank $k$, where $j^kg$ denotes the $k$-jet of $g$. 
\end{lem}


Now we can prove proposition \ref{prop:Sing}.

\begin{proof}
In the case of curves, $F:I\times\R^3 \to \R$  is given by
\begin{equation}\label{eq:FDist}
F(t,x)=\left[ \gamma'(t), \xi(t), x-\gamma(t) \right].
\end{equation}
Then $F_t= G+\tau_1^1F$, where
\begin{equation*}
G(t,x)=\left[ \eta(t), \xi(t), x-\gamma(t) \right].
\end{equation*}
Thus $F=F_t=0$ at $t=t_0$ if and only if $x=\gamma(t_0)+\lambda(t_0)\xi(t_0)$. Moreover $G_{t}=H-1-\mu F$, where
\begin{equation*}
H(t,x)=\sigma(t)\left[ \gamma'(t), \eta(t), x-\gamma(t) \right].
\end{equation*}
Thus $F=F_t=F_{tt}=0$ at $t=t_0$ if and only if $\sigma(t_0)\neq 0$ and $\lambda(t_0)=\sigma^{-1}(t_0)$. 
Differentiating again we obtain
\begin{equation}
H_t=\frac{\sigma_t}{\sigma}H-\tau_1^1H+\tau_2^1\sigma F.
\end{equation}
Thus $F=F_t=F_{tt}=0$ and $F_{ttt}\neq 0$ at $t=t_0$ if and only if $x=\gamma(t_0)+\sigma^{-1}(t_0)\xi(t_0)$ and $\sigma_t-\tau_1^1\sigma\neq 0$. In this case, $F$ has an $A_2$ singularity. Differentiating once more we obtain, at points where $F=F_t=F_{tt}=F_{ttt}=0$, 
\begin{equation*}
H_{tt}(t_0,x)=(\sigma_{t}-\sigma\tau_1^1)_t \sigma^{-1}(t_0).
\end{equation*}
We conclude that $F=F_t=F_{tt}=F_{ttt}=0$ and $F_{tttt}\neq 0$ at $t=t_0$ if and only if $x=\gamma(t_0)+\sigma^{-1}(t_0)\xi(t_0)$, $[\sigma_t-\sigma\tau_1^1](t_0)=0$ and $[\sigma_t-\sigma\tau_1^1]_t(t_0)\neq 0$. In this case,
$F$ has an $A_3$ singularity.

To complete the proof, we must prove that $F$ is a $\mathcal{R}$-versal unfolding of $f$. 
Observe that $F_x(t,x_0)=\gamma'(t)\times\xi(t)$, where $\times$ denotes vector product. For an $A_2$ point we write
$$
j^1F_x(t_0,x_0)=
\left[
\begin{array}{c}
\gamma'(t_0)\times\xi(t_0)\\
\gamma''(t_0)\times\xi(t_0)
\end{array}
\right],
$$
which has rank $2$. For an $A_3$ point we write
$$
j^2F_x(t_0,x_0)=
\left[
\begin{array}{c}
\gamma'(t_0)\times\xi(t_0)\\
\gamma''(t_0)\times\xi(t_0)\\
\sigma(t_0) \gamma'(t_0)\times\gamma''(t_0)
\end{array}
\right].
$$
Since $\sigma(t_0)\neq 0$, this matrix has rank $3$. By lemma \ref{lem:unfolding}, $F$ is a versal unfolding of a point $A_k$, $k=2,3$. 
\end{proof}

 \subsection{Realization of simple singularities of $ET_N$}
 
In this section, we give several examples of singularities that occur in $ET_N$. Through these examples, we show 
that any simple singularity can appear in $ET_N$. We recall that any simple singularity is $\mathcal{R}$-equivalent
to $A_k$, $k\geq 2$, $D_k$, $k\geq 4$, $E_6$, $E_7$ or $E_8$ (see \cite{Giblin}, ch.11). 

Consider $M\subset\R^{n+2}$ given by the graph of $f(t,y)$, $t=(t_1,...,t_n)$. Then $M$ is given by   
$$
\psi(t,y)=\left( t_1,....,t_n, y, f(t,y) \right).
$$
Thus 
\[
\psi_{t_i}=\left( e_i, 0, f_{t_i} \right);\ \ \psi_{y}=\left( 0, 1, f_{y} \right),
\]
where $e_i=(0,..,1,...0)$ with $1$ in the component $i$. We shall assume that $f=f_{t_i}=f_y=0$ at the origin, for any $1\leq i\leq n$.
Let $N$ be the submanifold $y=g(t)$ and assume that $g_{t_i}=0$ at $t=0$, i.e., the tangent plane of $N$ is generated by
$\{e_i\}$, $1\leq i\leq n$.

Let $x=(x_1,...,x_{n+2})$ and write the vector field $\xi$ as 
$$
\xi(t)=\sum_{i=1}^n a_i(t)\psi_{t_i}+\psi_y.
$$
Then $F(t,x)=\det\left( \psi_{t_1}(t),...,\psi_{t_n}(t),\xi(t),x-\psi(t) \right)$ can be written as
$$
F(t,x)=\det\left( \psi_{t_1}(t),...,\psi_{t_n}(t),\psi_y(t),x-\psi(t) \right).
$$
We conclude that 
$$
F=f-x_{n+2}+\sum_{i=1}^n f_{t_i}(x_i-t_i)+f_y(x_{n+1}-g),
$$
where $f,f_{t_i},f_y$ are taken at $(t,g(t))$.

\begin{lem}
Assume that $(f_{t_it_j}(0))$ is the identity matrix and $f_{yt_i}(0)=0$, for any $i$. Then $\psi_y$ is the Darboux direction at $0$. Moreover, the shape operator $S_1$ at the origin is given by $(f_{t_it_jy}(0))$. 
\end{lem}

\begin{proof}
First observe that $\psi_{yt_i}(0)=0$, for any $i$. This implies that $\psi_y$ is the Darboux direction at the origin. Moreover 
$$
\xi_{t_j}=\sum_{i=1}^n (a_i)_{t_j}\psi_{t_i}+ \sum_{i=1}^n a_i\psi_{t_it_j}+\psi_{yt_j}
$$
Since these vectors are tangent to $M$ and $\psi_{t_it_j},\psi_{yt_j}$ are co-linear and transversal to $M$, we obtain
\begin{equation}\label{eq1}
\xi_{t_j}=\sum_{i=1}^n (a_i)_{t_j}\psi_{t_i}. 
\end{equation}
This implies that
\begin{equation}\label{eq2}
\sum_{i=1}^n a_i f_{t_it_j}+f_{yt_j}=0
\end{equation}
Observe that $a_i=0$ at $0$. Differentiating equation \eqref{eq2}
and taking $t=0$ we obtain
$(a_i)_{t_j}(0,0)=-f_{t_it_jy}(0). $
Now equation \eqref{eq1} implies the second part of the lemma. 
\end{proof}

Lemma 3.4 explicitly provides the Darboux direction and calculates the shape operator $S_1$ at the origin, thus indicating the way to find the realization of simple singularities of the envelope of tangent spaces of $M$ along $N$. We shall now describe examples of functions $f(t,y)$ and $g(t)$ such that
the corresponding families $F(t,x)$ given by equation \eqref{eq:DefineFH} are versal unfoldings of functions $F(t,x_0)$, $x_0=(0,...,\sigma^{-1},0)\in ET_N$, 
with singular points of type $A_k$, $k\geq 2$, $D_k$, $k\geq 4$, $E_6$, $E_7$ and $E_8$ at $t=0$. In each of the following examples, $\sigma$ is eigenvalue of $S_1$, simple in case of $A_k$ and double in cases of $D_k$ and $E_k$. To simplify the formulas we have taken sometimes $\sigma = 1$.

\begin{exam} 
(1) Let
$$
f(t,y)=\frac{t^2}{2} +\frac{1}{6}t^3 +\frac{\sigma}{2}t^2y
$$
and $g(t)=0$. Then, close to $(0,\sigma^{-1},0)$, 
$$
F(t, x_1, x_2+\sigma^{-1},x_3)=-\frac{1}{3}t^3+\frac{\sigma}{2}t^2x_2+(\frac{1}{2}t^2+t)x_1-x_3,
$$
which is a versal unfolding of an $A_2$ point. 

\smallskip\noindent
(2) Let
$$
f(t,y)=\frac{t^2}{2} +\frac{1}{24}t^4 +\frac{\sigma}{2}t^2y
$$
and $g(t)=0$.  Close to $(0,\sigma^{-1},0)$, 
$$
F(t, x_1, x_2+\sigma^{-1},x_3)=-\frac{1}{8}t^4+\frac{\sigma}{2}t^2x_2+(\frac{1}{6}t^3+t)x_1-x_3,
$$
which is a versal unfolding of an $A_3$ point. 

\smallskip\noindent
(3) Let
$$
f(t_1,t_2,y)=\frac{1}{2}(t_1^2+t_2^2)+\frac{\sigma}{2}t_1^2y+ t_1^3t_2
$$
For $\sigma=1$, choose $g=-t_1^3-3t_1t_2$. Then, close to $(0,0,1,0)$,
$$
F=t_1^5-\frac{1}{2}t_2^2+(t_1-t_1^4)x_1+(t_1^3+t_2)x_ 2+\frac{1}{2}t_1^2x_3-x_4,
$$
which is a versal unfolding of an $A_4$ point.

\smallskip\noindent
(4) For general $k\geq 3$, let $\sigma=1$, $t=(t_1,...,t_{k-2})$, i.e., $n=k-2$, 
$$
f(t,y)=\frac{1}{2}|t|^2+\frac{1}{2}t_1^2y+\sum_{j=2}^{k-2} t_1^{j+1}t_{j},
$$
and
$$
g(t)=-t_1^{k-1}-\sum_{j=2}^{k-2} (j+1)t_1^{j-1}t_j
$$
Then, close to $(0,...,1,0)$, 
$$
F=t_1^{k+1}-\frac{1}{2}\sum_{j=2}^{k-2} t_j^2+x_1(t_1-t_1^k)+\sum_{j=2}^{k-2}x_j(t_j+t_1^{j+1})+\frac{1}{2}t_1^2x_{k-1}-x_k.
$$
which is a versal unfolding of an $A_k$ point.
\end{exam}

\begin{exam} 
(1) Let
$$
f(t_1,t_2,y)=\frac{1}{2}(t_1^2+t_2^2)+\frac{\sigma}{2}(t_1^2+t_2^2)y+ t_1^3+t_1t_2^2
$$
and $g=0$. Then 
$$
F=-2(t_1^3+t_1t_2^2)-x_4 +x_1(t_1+3t_1^2+t_2^2)+x_2(t_2+2t_1t_2)+\frac{\sigma}{2}(t_1^2+t_ 2^2)x_3
$$
which is a versal unfolding of a $D_4$ point.

\smallskip\noindent
(2) For a general $k\geq 4$, take 
$$
f=\frac{1}{2}|t|^2+\frac{y}{2}(t_1^2+t_2^2)+t_1^{k-1}+t_1t_2^2+\sum_{j=3}^{k-2}t_1^jt_j+\sum_{j=3}^{k-2}t_1^{j-2}t_2^2t_j
$$
and 
$g=-\sum_{j=3}^{k-2}jt_jt_1^{j-2}.$
Long but straightforward calculations show that, close to $(0,...,1,0)$, 
$$
F=(2-k)t_1^{k-1}-2t_1t_2^2-\frac{1}{2}\sum_{j=3}^{k-2}t_j^2-x_k+\frac{1}{2}(t_1^2+t_2^2)x_{k-1}
$$
$$
+\sum_{j=3}^{k-2}x_j(t_1^j+t_1^{j-2}t_2^2+t_j)+x_2\left(  t_2+2t_1t_2+\sum_{j=3}^{k-2}(2-j)t_1^{j-2}t_2t_j \right) 
$$
$$
+x_1\left( t_1+ (k-1)t_1^{k-2}+t_2^2+\sum_{j=3}^{k-2}(j-2)t_1^{j-3}t_2^2t_j \right),
$$
which is a versal unfolding of a $D_k$ point.
\end{exam}

\begin{exam} 
(1) Consider
$$
f=\frac{1}{2}|t|^2+\frac{1}{2}(t_1^2+t_2^2)y+t_1^3+t_2^4+t_1t_2t_3+2t_1t_2t_3y+t_1t_2^2t_4+3t_1t_2^2t_4y
$$
and $g=0$. Then 
$$
F=-2t_1^3-3t_2^4-\frac{1}{2}(t_3^2+t_4^2)-x_6+x_4(t_1t_2^2+t_4)+x_3(t_1t_2+t_3)
$$
$$
+x_1\left( t_1+3t_1^2+t_2^2t_4+t_2t_3\right)+x_2\left( t_2+4t_2^3+t_1t_3+2t_1t_2t_4 \right)
$$
$$
+x_5\left(  \frac{1}{2}(t_1^2+t_2^2)+2t_1t_2t_3+3t_1t_2^2t_4\right)
$$
which is a versal unfolding of an $E_6$ point.

\smallskip\noindent
(2) Let
$$
f=\frac{1}{2}|t|^2+\frac{1}{2}(t_1^2+t_2^2)u+t_1^3+t_1t_2^3+t_1t_2t_3+2t_1t_2t_3u+t_1^2t_2t_4+3t_1^2t_2t_4u+t_2^2t_5+2t_2^2t_5u
$$
and $g=0$. Then
$$
F=-2t_1^3-3t_1t_2^3-\frac{1}{2}(t_3^2+t_4^2+t_5^2)-x_7+x_5(t_2^2+t_5)+x_4(t_1^2t_2+t_4)+x_3(t_1t_2+t_3)
$$
$$
+x_1\left( t_1+2t_1t_2t_4+t_2^3+3t_1^2+t_2t_3  \right)+x_2\left( t_2+3t_1t_2^2+2t_2t_5+t_1t_3+t_1^2t_4 \right)
$$
$$
+x_6\left( \frac{1}{2}(t_1^2+t_2^2)+3t_1^2t_2t_4+2t_1t_2t_3+2t_2^2t_5 \right),
$$
which is a versal unfolding of an $E_7$ point.

\smallskip\noindent
(3) Let 
$$
f=\frac{1}{2}|t|^2+\frac{1}{2}(t_1^2+t_2^2)u+t_1^3+t_2^5+t_1t_2t_3+2t_1t_2t_3u
$$
$$
+t_1t_2^2t_4+3t_1t_2^2t_4u+t_2^3t_5+3t_2^3t_5u+t_1t_2^3t_6+4t_1t_2^3t_6u
$$
and $g=0$. Then 
$$
F=-2t_1^3-4t_2^5-\frac{1}{2}(t_3^2+t_4^2+t_5^2+t_6^2)+x_4(t_1t_2^2+t_4)+x_3(t_1t_2+t_3)+x_5(t_2^3+t_5)+x_6(t_1t_2^3+t_6)
$$
$$
+x_1\left( t_1+3t_1^2+t_2^2t_4+t_2t_3+t_2^3t_6\right)+x_2\left( t_2+5t_2^4+t_1t_3+2t_1t_2t_4+3t_2^2t_5+3t_1t_2^2t_6 \right)
$$
$$
+x_7\left(  \frac{1}{2}(t_1^2+t_2^2)+2t_1t_2t_3+3t_1t_2^2t_4+4t_1t_2^3t_6+3t_2^3t_5\right)-x_8,
$$
which is a versal unfolding of an $E_8$ point.
\end{exam}

\section{Affine metrics and normal plane bundles}\label{sec:VectorFields}

\subsection{Affine metric of a vector field}

Fix a vector field $\xi$ in the osculating Darboux direction. For a local frame $\{X_1,...,X_n\}$ of $TN$ and $X,Y\in TN$, define
$$
G_{\xi}(X,Y)=[X_1,...,X_n, D_XY, \xi].
$$
It is proved in \cite{LSan} that
\begin{equation}\label{eq:Metric}
g_{\xi}(X,Y)=\frac{G_{\xi}(X,Y)}{{\det G_{\xi}(X_i,X_j)}^{\frac{1}{n+2}} }
\end{equation}
is a metric in $N$ (see also \cite{LSanTese}, ch.6). Assuming $[X_1,...,X_n,\eta,\xi]=1$, we get that
$G_{\xi}(X_i,X_j)=h^2(X_i,X_j)$. Thus the non-degeneracy hypothesis of the matrix $(h^2(X_i,X_j))$ implies 
that the metric $g_{\xi}$ is also non-degenerate. 

\subsection{Affine normal plane bundle}\label{sec:AffineNormalBundle}

Assume for the moment that we have chosen a transversal bundle $\sigma_1$ generated by $\{\xi,\bar\eta_1\}$. Take a $g_{\xi}$-orthonormal frame $\{X_1,...,X_n\}$ of the tangent space of $N$ and change the basis of $\sigma_1$ by writing $\eta_1=\mu\bar\eta_1+\lambda\xi$. By choosing an adequate $\mu$, we may assume that $[X_1,...,X_n,\eta_1,\xi]=1$. Note that
$\lambda$ remains arbitrary and $h^2(X_i,X_j)=\delta_{ij}$.

Now we shall make a particular choice for the transversal bundle. Write
\begin{equation*}
\eta=\eta_1-\sum_{k=1}^n\tau_2^2(X_k)X_k.
\end{equation*}
Direct computations show that $D_{X_i}\eta$ is tangent to $M$, for $1\leq i\leq n$, and so $D_X\eta$ is tangent to $M$, for any $X$ tangent to $N$. The transversal bundle $\sigma$ generated by $\{\xi,\eta\}$ is called the {\it affine normal plane bundle}.

It is proved in (\cite{LSan}, props.3.5 and 3.6) that the affine normal plane bundle $\sigma$ is independent of the choice of the initial bundle $\sigma_1$ and the $g_{\xi}$-orthonormal basis
$\{X_1,...,X_n\}$ of the tangent space of $N$. Thus $\sigma$ depends only on the choice of the vector field $\xi$. We shall denote it by $\sigma=\sigma(\xi)$. The results of this section
are summarized in next proposition:

\begin{prop}
Given a codimension $1$ submanifold $N\subset M$ of a hypersurface $M\subset\R^{n+2}$ and a vector field $\xi$ in the osculating Darboux direction, equation \eqref{eq:Metric} defines
a metric $g_{\xi}$ in $N$. There exists a vector field $\eta$ transversal to $M$ such that
\begin{equation}\label{eq:AfNormal1}
[X_1,...X_n,\eta,\xi]=1,
\end{equation}
\begin{equation}\label{eq:AfNormal2}
h^2(X_i,X_j)=\delta_{ij},
\end{equation}
for any $g_{\xi}$-orthonormal frame $\{X_1,..,X_n\}$ of $N$, and
\begin{equation}\label{eq:AfNormal3}
\tau_1^2=\tau_2^2=0.
\end{equation}
\end{prop}

\bigskip\noindent
The transversal vector field $\eta$ satisfying equations \eqref{eq:AfNormal1}, \eqref{eq:AfNormal2} and \eqref{eq:AfNormal3} is not unique. In fact, any
vector field $\bar\eta$ of the form
\begin{equation}\label{eq:BarEta}
\bar\eta=\eta+\lambda\xi,
\end{equation}
for some scalar function $\lambda$, also satisfies these equations. Conversely, any vector field $\bar\eta$ satisfying equations \eqref{eq:AfNormal1}, \eqref{eq:AfNormal2} and \eqref{eq:AfNormal3} is given by equation \eqref{eq:BarEta}, for some scalar function $\lambda$.

\subsection{Blaschke metric and affine normal}

In this section we assume that $M\subset\R^{n+2}$ is non-degenerate (see remark \ref{rem:BlaschkeHyp}). Denote by
$\zeta$ the affine Blaschke vector field of $M\subset\R^{n+2}$ and by $h$ the Blaschke metric of $M$ (\cite{Nomizu}).
We shall investigate the conditions under which $\zeta$ belongs to the affine normal plane (see \cite{LSanTese}, th.5.15).

\begin{lem}\label{lem:Equivalences}
Assume that $\xi$ is a vector field in the osculating Darboux direction and let $\zeta$ be the affine Blaschke vector field of $M\subset\R^{n+2}$. Let $\{X_1,...,X_n\}$ be an $h$-orthonormal
local frame of $TN$. The following conditions are equivalent:
\begin{enumerate}
\item $h(\xi,\xi)=1$.
\item $\{X_1,...,X_n,\xi\}$ $h$-orthonormal local frame of $TM$.
\item $[X_1,....,X_n,\zeta,\xi]=1$.
\item $\{X_1,...,X_n\}$ $g_{\xi}$-orthonormal local frame of $TN$.
\item The metric $g_{\xi}$ is the restriction of the Blaschke metric of $M$ to $N$.
\item $\zeta$ is contained in the affine normal plane bundle.
\end{enumerate}
\end{lem}
\begin{proof}
It is easy to verify the equivalence between items 1,2 and 3. It is also easy to verify the equivalence of 3 and 4, while the equivalence between 4 and 5 is obvious.
Finally, since equations \eqref{eq:AfNormal2} and \eqref{eq:AfNormal3} always hold, item 6 is equivalent to equation \eqref{eq:AfNormal1}, thus to item 3.
\end{proof}

\section{Parallel vector fields}

Consider a vector field $\xi$ in the osculating Darboux direction of $N\subset M$. We say that $\xi$ is {\it parallel} if
$D_X\xi$ is tangent to $N$, for any $X$ tangent to $N$.

\subsection{Darboux frames for curves}\label{sec:Curves}

Consider a smooth curve $\gamma\subset M$, where $M\subset\R^3$ is a surface. We shall assume that the osculating plane of $\gamma$ does not coincide with the tangent plane of $M$.

Let $\gamma(t)$, $t\in I$, be a curve contained in a surface $M$. We say that the parameterization is {\it adapted} to $M$ if $\gamma'''(t)\in T_{\gamma(t)}M$, for each $t\in I$.
Observe that when $\gamma$ is contained in a hyperplane, the affine parameterization $\gamma(s)$ satisfies $\gamma'''(s)=-\mu(s)\gamma'(s)$, where $\mu(s)$ is the affine curvature of $\gamma$, and so this parameterization is adapted to $M$.

\begin{lem}
Assume that $\gamma''(s)\not\in T_{\gamma(s)}M$, for each $s\in I$, i.e., the osculating plane of $\gamma$ does not coincide with the tangent plane of $M$.
Then $\gamma$ admits an unique, up to linear changes, adapted re-parameterization.
\end{lem}
\begin{proof}
Let $t$ be a new parameter with the change of variables given by $s=s(t)$. Then
\begin{align*}
\gamma_t& =\gamma_s s_t,\\
\gamma_{tt}&=\gamma_{ss} s_t^2+ \gamma_s s_{tt},\\
\gamma_{ttt}&=\gamma_{sss}s_t^3+ 3 \gamma_{ss}s_t s_{tt}+ \gamma_s s_{ttt}.
\end{align*}
Let $\nu$ be a co-normal vector field of $M$, i.e., $\nu(Z)=0$, for any $Z$ tangent to $M$.
Then $\gamma_{ttt}\in T_{\gamma(t)} M$ if and only if $\nu(\gamma_{ttt})=0$, which is equivalent to solve the following differential equation
\[
\nu(\gamma_{sss})s_t^3+3\nu(\gamma_{ss})s_t s_{tt}=0.
\]
Since $s_t>0$, this equation is equivalent to 
\[
A s_t^2+3 B s_{tt}=0,
\]
where $A=\nu(\gamma_{sss})$ and $B=\nu(\gamma_{ss})$. Since $B\neq 0$ by hypothesis, the lemma follows.
\end{proof}

Assume now that $\gamma(t)$ is an adapted parameterization of $\gamma$ and let $\xi$ be a multiple of $\xi_0$ satisfying $[\gamma'(t),\gamma''(t),\xi(t)]=1$.
Differentiating this equation we obtain that $\xi'(t)$ is tangent to $\gamma$, and so $\xi$ is parallel. We shall call the frame $\{\gamma'(t),\gamma''(t),\xi(t)\}$
the {\it affine Darboux frame} of $\gamma\subset M$. The structural equations of this frame are given by
\begin{equation*}
\left\{
\begin{array}{c}
(\gamma'(t))'=\gamma''(t),\\
(\gamma''(t))'=-\mu(t)\gamma'(t)+\tau(t)\xi(t),\\
\xi'(t)=-\sigma(t)\gamma'(t).
\end{array}
\right.
\end{equation*}

\subsection{Equiaffine transversal bundles}

Fix a vector field $\xi$ in the osculating Darboux direction and let $\sigma=\sigma(\xi)$ be the affine normal transversal bundle described in section \ref{sec:VectorFields}.
Then the corresponding connection $\nabla$ given by equation \eqref{eq:Decomposition} depends also on the choice of $\xi$. Thus we write $\nabla=\nabla(\xi)$.

Consider the volume form $\omega_g$ induced by the metric $g=g_{\xi}$. We say that the pair $(\nabla,g)$ is {\it equiaffine} if $\nabla(\omega_g)=0$. Let
$\{X_1,...,X_n\}$ be a $g$-orthonormal local frame of $TN$. 
Since
$$
\nabla_{X_i}\omega_g=-\sum_{j=1}^n \omega_g(X_1,..,\nabla_{X_i}X_j,...,X_n)=-\sum_{j=1}^n \Gamma_{ij}^j,
$$
where $\Gamma_{ij}^k$ are the Christoffel symbols of the connection, we conclude that $(\nabla,g)$ is equiaffine if and only if
\begin{equation}\label{eq:Equiaffine}
\Gamma_{i1}^1+\Gamma_{i2}^2+...+\Gamma_{in}^n=0,\ \ 1\leq i\leq n.
\end{equation}

\begin{prop}
The pair $(\nabla(\xi),g_{\xi})$ is equiaffine if and only if $\xi$ is parallel.
\end{prop}
\begin{proof}
Differentiating $[X_1,...,X_n,\eta,\xi]=1$ in the direction $X_i$ we obtain
$$
\sum_{k=1}^n\Gamma_{ik}^k+\tau_1^1(X_i)=0.
$$
Since $\xi$ is parallel if and only if $\tau_1^1(X_i)=0$, for $1\leq i\leq n$, the proposition is proved.
\end{proof}

\subsection{The apolarity condition}

The cubic forms are defined as
\begin{equation}
\begin{array}{c}
C^1(X,Y,Z)=(\nabla_Xh^1)(Y,Z)+\tau_1^1(X)h^1(Y,Z)+\tau_2^1(X)h^2(Y,Z),\\
C^2(X,Y,Z)=(\nabla_Xh^2)(Y,Z)+\tau_1^2(X)h^1(Y,Z)+\tau_2^2(X)h^2(Y,Z).
\end{array}
\end{equation}
One can verify that the cubic forms are symmetric in $X,Y,Z$.

Take any $\eta$ in the affine normal plane described in section \ref{sec:AffineNormalBundle}. In this case,
$h^2$ coincides with $g_{\xi}$ and the cubic form $C^2$ can be written as
\begin{equation}
C^2(X,Y,Z)=(\nabla_Xh^2)(Y,Z).
\end{equation}
The cubic form $C^2$ is {\it apolar} with respect to $h^2$ if
\begin{equation}
tr_{h^2}C^2(X,\cdot,\cdot)=0,
\end{equation}
for any vector field $X\in TN$.

Consider a $g_{\xi}$-orthonormal basis $\{X_1,...,X_n\}$. Then the apolarity condition for $(C^2,h^2)$
is equivalent to equation \eqref{eq:Equiaffine}. Thus we conclude the following proposition:

\begin{prop}\label{prop:EquivApolar}
The following statements are equivalent:
\begin{enumerate}
\item The cubic form $C^2$ is apolar with respect to $h^2$.
\item The pair $(\nabla,g)$ is equiaffine.
\item The vector field $\xi$ is parallel.
\end{enumerate}
\end{prop}

\subsection{Examples}

We now give some examples of submanifolds $N\subset M$ that admit a parallel vector field $\xi$.

\begin{exam}\label{ex:Curvas}
When $n=1$, the vector field $\xi$ defined in section \ref{sec:Curves} is parallel. Assume that $\gamma(t)$ is an adapted parameterization. Then
$\gamma'(t)$ is $g_{\xi}$-unitary and $\{\xi(t),\gamma''(t)\}$ is a basis for the affine normal plane $\sigma(\xi)$. 
\end{exam}

\begin{exam}\label{ex:Hyperplanar} {\it Hyperplanar submanifolds}.
Assume $N=M\cap H$, where $H$ is a hyperplane.  Fix a point $p_0\in N$ and let $\xi(p_0)$ be a vector in the osculating Darboux direction at $p_0$. We can extend $\xi$ to a vector field along $N$ in the osculating Darboux direction such that $\xi(p)=\xi(p_0)+\e(p)$, where $\e(p)\in H$, for any $p\in N$. Then $D_X\xi\in H\cap T_pM=T_pN$. The metric $g_{\xi}$ defined by equation \eqref{eq:Metric} coincides with the Blaschke metric of $N\subset H$
and the affine Blaschke normal $\zeta$ of $N\subset H$ belongs to the affine normal plane.
\end{exam}

\begin{exam} {\it Visual contour submanifolds} (\cite{Cipolla}). Suppose that all lines $\xi_0$ meet at a point $O$. Then we can choose $\xi=\lambda\xi_0$
such that $\xi(p)=O-p$. Differentiating we obtain $-S_1X+\tau_1^1(X)\xi=-X$. We conclude that $S_1=Id$ and $\tau_1^1=0$. Thus $\xi$ is parallel.
\end{exam}

\begin{exam}\label{ex:Hyperquadric}
Suppose $M$ is a hyperquadric and $N\subset M$ is arbitrary. Using the notation of remark \ref{rem:BlaschkeHyp}, we have that $h(\xi,X)=0$, for any $X\in T_pN$.
Take $\xi$ such that $h(\xi,\xi)=1$. Differentiating and using that the cubic form $C$ of $M\subset\R^{n+2}$  is zero we get $h(\tilde\nabla_X\xi,\xi)=0$. Thus
$\tilde\nabla_X\xi\in T_pN$ and so $\xi$ is parallel. It is not difficult to see that in this case the affine normal Blaschke vector field is contained in the affine normal
plane bundle and the metric $g_{\xi}$ is the restriction of the Blaschke metric of $M$ to $N$ (see \cite{Nomizu}). 
\end{exam}

\subsection{Blaschke metric and the parallelism condition}

In this section we shall assume that  $M\subset\R^{n+2}$ is non-degenerate (see remark \ref{rem:BlaschkeHyp}). Denote by
$\zeta$ the affine Blaschke vector field of $M\subset\R^{n+2}$ and by $h$ the Blaschke metric of $M$. We shall also assume that $\xi$ is parallel 
and look for conditions under which $\zeta$ belongs to the affine normal plane.

\begin{prop}\label{prop:BlaschkeMetric}
Let $\xi$ be a parallel vector field along $N$ and assume that $h(\xi,\xi)=1$ at a certain point $p_0\in N$.
The conditions of lemma \ref{lem:Equivalences} hold if and only if $C(X,\xi,\xi)=0$, for any $X\in TN$.
\end{prop}
\begin{proof}
Assume that $h(\xi,\xi)=1$. Differentiating this equation in the direction $X\in TN$ and using that $\xi$ is parallel we obtain
$C(X,\xi,\xi)=\tilde\nabla(X,\xi,\xi)=0$, for any $X\in TN$. Conversely, if $C(X,\xi,\xi)=0$ we obtain $X(h(\xi,\xi))=0$, for any $X\in TN$. Thus
$h(\xi,\xi)=1$ at $N$.
\end{proof}

\begin{exam}
Consider the surface $M\subset\R^3$ graph of 
$$
f(x,y)=\frac{1}{2}(x^2+y^2)+\frac{c}{6}(x^3-3xy^2), 
$$
where $c\neq 0$. Let $\gamma=N$ be the intersection of $M$ with the plane $y=0$. Take $\xi$ parallel as in section \ref{sec:Curves}. Then the affine 
plane is generated by $\{\xi,\eta\}$, where $\eta$ is the affine normal of $\gamma$ (see example \ref{ex:Curvas}). Since $\eta(0,0)=(-c/3,0,1)$, $\xi(0,0)=(0,1,0)$ and 
$\zeta(0,0)=(0,0,1)$, where $\zeta$ denotes the affine Blaschke normal of $M$, we conclude that, if $c\neq 0$, $\zeta$ does not belong to the affine normal plane.
\end{exam}

\section{Transon planes}

\subsection{Transon planes for curves}

The following theorem is a very old result (\cite{Transon},\cite{Juttler}):

\begin{thm}\label{thm:TransonCurves}
Given a surface $M\subset\R^3$, a point $p_0\in M$ and a vector $T\in T_{p_0}M$, consider sections of $M$ by planes $H$ containing $T$ and passing through $p_0$.
Then the affine normal vectors $\eta=\eta(H)$ of these sections at $p_0$ belong to a plane. This plane is called the {\it Transon plane} of $p_0\in M$ in the direction $T$.
\end{thm}

The statement of this theorem needs an explanation: A parameterization $\gamma(s)$ of $\gamma=M\cap H$ by affine arc-length is defined by the condition
$[\gamma'(s),\gamma''(s), \xi_0]=1$, for some constant vector field $\xi_0$. Then the affine normal at $p_0=\gamma(s_0)$ is just $\gamma''(s_0)$. We remark
that, instead of a constant vector field $\xi_0$, we may also consider here a vector field $\xi$ in the osculating Darboux direction of $\gamma\subset M$ such that
$\xi(p)=\xi(p_0)+\e(p)$, $\e(p)\in H$, for any $p\in\gamma$, where
$\xi(p_0)$ is any vector in the osculating Darboux direction of $\gamma\subset M$ at $p_0$.

\subsection{Transon planes in arbitrary dimensions}

Let $M\subset\R^{n+2}$ be a hypersurface, $p_0\in M$ and $T$ a $n$-dimensional subspace
contained in $T_{p_0}M$. For a hyperplane $H$ containing $T$, consider the vector field $\xi$
along $N=M\cap H$ given by $\xi(p)=\xi(p_0)+\e(p)$, $\e(p)\in H$, as in example \ref{ex:Hyperplanar}.
The metric $g_{\xi}$ in $N$ is, up to a constant, the Blaschke metric of $N$, and if we choose the unique $\eta\in H$ satisfying equations \eqref{eq:AfNormal1}, \eqref{eq:AfNormal2},
\eqref{eq:AfNormal3} and \eqref{eq:BarEta}, then $\eta$ is the affine Blaschke normal of $N\subset H$.

Next theorem says that, as in the case of curves, this affine normal plane at $p_0$ is independent of the hyperplane $H$, and we shall keep the name Transon plane for it.

\begin{thm}\label{thm:Transon}
Given a hypersurface $M\subset\R^{n+2}$, a point $p_0\in M$ and a $n$-subspace $T\subset T_{p_0}M$, 
consider sections of $M$ by hyperplanes $H$ containing $T$ and passing through $p_0$.
Assume that one section $H_0\cap M$ is non-degenerate in the sense of section \ref{sec:Xi}. Then $H\cap M$ is non-degenerate for any $H$ and
the affine normal vectors $\eta=\eta(H)$ of these sections at $p_0$ belong to a plane.
\end{thm}

\begin{proof}
Assume that $p_0$ is the origin, the tangent to $M$ is the plane $z=0$ and that $H_0$ is the hyperplane $y=0$. The non-degeneracy hypothesis implies that we can
find a local coordinate system such that $M$ is given by
\begin{equation}
z=\frac{1}{2}(x_1^2+...+x_n^2+ay^2)+P_3(x)+yP_2(x)+y^2P_1(x)+P_0y^3+O(4)(x,y).
\end{equation}
where $a\in\R$, $P_k(x)$ is homogeneous of degree $k$ in $x=(x_1,...,x_n)$ and $O(4)(x,y)$ denote terms of degree $\geq 4$ in $(x,y)$.
Consider the hyperplane $H_{\lambda}$ of equation $y=\lambda z$. The projection
of these sections in the $xz$-hyperplane is given by
\begin{equation}
z=\frac{1}{2}(x_1^2+...+x_n^2+a\lambda^2z^2)+P_3(x)+\lambda zP_2(x)+\lambda^2z^2P_1(x)+P_0\lambda^3z^3+O(4)(x,z),
\end{equation}
where $O(4)(x,z)$ denote terms of degree $\geq 4$ in $(x,z)$. This curve can be re-written as
\begin{equation}
z=\frac{1}{2}(x_1^2+...+x_n^2)+P_3(x)+O(4)(x).
\end{equation}
where $O(4)(x)$ denotes terms of degree $\geq 4$ in $x$.
This implies that the projection of the affine normal vector does not depend on $\lambda$.
\end{proof}

\subsection{Transon planes for general submanifolds}

\begin{lem}\label{lemma:lema1}
Let $\bar{N}$ be the image of $N$ by the projection $\pi:\R^{n+2}\to H$ in a hyperplane $H$ along the constant direction $\xi(p_0)$.
The Blaschke affine normal $\tilde\eta$ of $\bar{N}$ at $p_0$ belongs to the Transon plane.
\end{lem}

\begin{proof}
We may assume that $N$ is defined by $y=P_2(x,z)+O(3)$. Then the same argument as above proves the proposition.
\end{proof}

Denote by $\bar\eta$ the vector field along $\bar{N}$ such that $\bar\eta(\pi(x))=\eta(x)$ and by $\bar{X}=\pi_*(X)$ the projection of $X$ in $H$.
Let $\eta$ be the transversal vector field along $N$ in the affine normal plane bundle such that $\eta$ is parallel to $H$.

\begin{lem}\label{lemma:bar}
We have that
\begin{enumerate}
\item $\bar{h}^2(X,Y)(p_0)=h^2(X,Y)(p_0)$, for any $X,Y$ tangent to $N$.
\item $\bar\nabla(\bar{h}^2)(X,Y,Z)(p_0)=\nabla(h^2)(X,Y,Z)(p_0)$, for any $X,Y,Z$ tangent to $N$.
\end{enumerate}
Thus $\eta$ is apolar if and only if $\bar\eta$ is apolar.
\end{lem}
\begin{proof}
Write $\pi(p)=p+\lambda(p)\xi(p_0)$. Then $\pi_*X=X+X(\lambda)\xi(p_0)$, which implies that $X(\lambda)(p_0)=0$. Differentiating again we obtain
$$
D_Y\pi_*X=D_YX+YX(\lambda)\xi+X(\lambda)D_Y\xi.
$$
This implies that $\bar{h}^2(\bar{X},\bar{Y})=h^2(X,Y)$ for any $p$ and $\bar\nabla_{\bar{Y}}\bar{X}=\nabla_YX$ at $p_0$. If $\{X_1,...,X_n\}$ is a $h^2$-orthonormal frame, then
$\{\bar{X}_1,...,\bar{X}_n\}$ is a $\bar{h}^2$-orthonormal frame. Thus 
$$
\bar\nabla_{\bar{X}_k}\bar{h}^2(\bar{X}_i,\bar{X}_j)=-\bar{h}^2(\bar\nabla_{\bar{X}_k}\bar{X}_i,\bar{X}_j)-\bar{h}^2(\bar\nabla_{\bar{X}_k}\bar{X}_j,\bar{X}_i)
$$
is equal, at $p_0$, to
$$
\nabla_{{X}_k}{h}^2({X}_i,{X}_j)=-{h}^2(\nabla_{{X}_k}{X}_i,{X}_j)-{h}^2(\nabla_{{X}_k}{X}_j,{X}_i),
$$
thus proving the lemma.
\end{proof}

Next lemma is a general result concerning codimension $1$ immersions:

\begin{lem}\label{lemma:BlaschkeNormal}
Let $S\subset\R^{n+1}$ be a hypersurface, $p_0\in S$, and let $\zeta$ be a transversal vector field such that $D_X\zeta$ at $p_0$ is tangent to $S$, for any
vector field $X$ tangent to $S$. If $\zeta$ is apolar at $p_0$, then $\zeta(p_0)$ is a multiple of the affine normal vector. Conversely,
if $\zeta(p_0)$ is a multiple of the affine normal vector, then $\zeta$ is apolar at $p_0$.
\end{lem}

\begin{proof}
Let $\tilde\zeta$ denote the affine Blaschke normal vector field and 
$$
\zeta=\sum_{l=1}^na_lX_l+\delta\tilde\zeta.
$$
Writing 
$$
D_{X_i}X_j=\nabla_{X_i}X_j+h(X_i,X_j)\zeta=\tilde\nabla_{X_i}X_j+\tilde{h}(X_i,X_j)\zeta,
$$
we conclude that $\tilde{h}(X_i,X_j)=\delta {h}(X_i,X_j)$ and
$$
\tilde\Gamma_{ij}^l =\Gamma_{ij}^l+h(X_i,X_j)a_l.
$$
Moreover, since $D_{X_k}\zeta$ is tangent to $S$ at $p_0$, we conclude that, at $p_0$,
$$
X_k(\delta)=-\delta\sum_{l=1}^n a_lh(X_k,X_l).
$$
Choose a basis $\{X_1,..,X_n\}$ $h$-orthonormal. Then 
$$
\tilde{C}(X_k,X_i,X_i)=X_k(\delta)-2\delta\tilde{\Gamma}_{ik}^i,\ \ C(X_k,X_i,X_i)=-2\Gamma_{ik}^i.
$$
Since $\tilde{C}$ is apolar,
$$
0=\sum_{i=1}^n C(X_k,X_i,X_i)-(n+2)\delta a_k.
$$
Thus $C$ is apolar at $p_0$ if and only if $a_k(p_0)=0$, for $1\leq k\leq n$, which is equivalent to to $\zeta$ multiple of $\tilde\zeta$.
\end{proof}

\begin{thm}
The affine normal plane coincides with the Transon plane if and only if $\xi$ is parallel.
\end{thm}

\begin{proof} 
By proposition \ref{prop:EquivApolar}, $\xi$ is parallel at $p_0$ if and only if $\eta$ is apolar at $p_0$. By lemma \ref{lemma:bar}, $\eta$ is apolar at $p_0$ if and only if 
$\bar\eta$ is apolar at $p_0$. From lemma \ref{lemma:BlaschkeNormal}, $\bar\eta$ is a multiple of $\tilde\eta$ at $p_0$ if and only if $\bar\eta$
is apolar at $p_0$. Finally $\tilde\eta$ belongs to the Transon plane, by lemma \ref{lemma:lema1}.
\end{proof}

\section{Existence of parallel vector fields}

\subsection{Submanifolds that admit parallel vector fields}

In this section we characterize the submanifolds $N$ that admit a parallel vector field $\xi$. We begin with the following lemma:

\begin{lem}\label{lemma:TauExact}
There exists a parallel vector field $\xi$  if and only if $\tau_1^1$ is exact.
\end{lem}
\begin{proof}
Fix a vector field $\xi_0$ in the osculating Darboux direction and look for $\lambda$ such that $\xi=\lambda\xi_0$ is parallel.
Differentiating this equation we obtain
$$
D_X(\xi)=X(\lambda)\xi_0 +\lambda\left( -SX+\tau_1^1(X)\xi_0\right)=-\lambda SX+(X(\lambda)+\lambda\tau_1^1(X))\xi_0.
$$
Then $X(\lambda)+\lambda\tau_1^1(X)=0$, for any $X\in T_pN$, if and only if $\tau_1^1=-d(\log(\lambda))$.
\end{proof}

Since $D_X\xi$ is tangent to $M$, we can write
\begin{equation}\label{eq:NormalConnection}
D_X\xi=-S_1X+\nabla^{\perp}_X\xi,
\end{equation}
where $\nabla^{\perp}_X\xi=\tau_1^1(X)\xi$ is the affine normal connection.

Consider the normal bundle connection $\nabla^{\perp}$ defined by equation \eqref{eq:NormalConnection}.
The normal curvature $R^{\perp}$ is defined as
$$
R^{\perp}(X,Y)\xi=\nabla^{\perp}_Y\nabla^\perp_X\xi-\nabla^{\perp}_X\nabla^\perp_Y\xi+\nabla^{\perp}_{[X,Y]}\xi.
$$
We say that the normal bundle is {\it flat} if $R^{\perp}=0$, for any $X,Y\in T_pN$, $\xi=\lambda\xi_0$  (see \cite{doCarmo}, ch.6).

\begin{prop}
There exists a parallel vector field $\xi$ if and only the normal bundle is flat.
\end{prop}
\begin{proof}
Observe that $\nabla^{\perp}_X\xi_0=\tau_1^1(X)\xi_0$. Thus
$$
\nabla^{\perp}_Y\nabla^\perp_X\xi=Y\tau_ 1^1(X)\xi_0+\tau_ 1^1(X)\tau_1^1(Y)\xi_0.
$$
Now straightforward calculations shows that
$$
R^{\perp}(X,Y)\xi_0=\left( Y\tau_1^1(X)-X\tau_1^1(Y)+\tau_1^1([X,Y]) \right) \xi_0=d\tau_1^1(X,Y)\xi_0.
$$
Using lemma \ref{lemma:TauExact} we prove the proposition.
\end{proof}

\subsection{Example: A submanifold without a parallel vector field.}

We now give an explicit example of an immersion $N\subset M$ that does not admit a parallel vector field.

\begin{exam}
Take $M$ to be the graph of
$$
f(x_1,x_2,y)=\frac{1}{2}\left( x_1^2+x_2^2+y^2 \right) +\frac{1}{2}\left( k_1x_1^2y+k_2x_2^2y   \right)
$$
and $N$ given by the intersection of $M$ with $y=x_1x_2$. Thus $N$ can be parameterized by
$$
\phi(x_1,x_2)=\left( x_1, x_2, x_1x_2, f(x_1,x_2,x_1x_2)  \right).
$$
Let $X_1=\phi_{x_1}$ and $X_2=\phi_{x_2}$.
Then the vector field $\xi(x_1,x_2)=(\xi_1,\xi_2,\xi_3,\xi_4)$,
\begin{align*}
\xi_1&=        k_2^2x_2^3-k_2 x_1x_2^2 -2k_1k_2x_1^2x_2 -k_1x_1-x_2      \\
\xi_2&=    k_1^2x_1^3-k_1 x_1^2x_2 -2k_1k_2x_1x_2^2 -k_2x_2-x_1       \\
\xi_3&=      1+2(k_1+k_2)x_1x_2+3k_1k_2x_1^2x_2^2          \\
\xi_4&= -\frac{1}{2}\left( k_1x_1^2+2x_1x_2+k_2x_2^2 \right)+(k_1^2-k_1k_2)x_1^3x_2+(k_2^2-k_1k_2)x_1x_2^3 \\
&+\frac{1}{2}k_1k_2x_2^2x_1^2\left( k_1x_1^2+2x_1x_2+k_2x_2^2 \right)
\end{align*}
is tangent to $M$ and $D_{X_i}\xi\in T_pM$, for $i=1,2$.
Moreover,
\begin{align*}
\tau_1^1(X_1)(x_1,x_2)&=x_1+(3k_1+2k_2)x_2+O(3)\\
\tau_1^1(X_2)(x_1,x_2)&=(2k_1+3k_2)x_1+x_2+O(3)
\end{align*}
We conclude that $d\tau_1^1(X_1,X_2)(0,0)=k_1-k_2\neq 0$ and thus $\tau$ is not a closed form.
\end{exam}
\bibliographystyle{amsplain}

\end{document}